\documentclass[11pt]{amsart}
\usepackage{amsmath,amsfonts,amsthm,amssymb,amscd,color}
\usepackage{graphicx,epsf}
\usepackage[bookmarks]{hyperref}
\reversemarginpar

\newtheorem{theorem}{Theorem}[section]
\newtheorem{corollary}[theorem]{Corollary}

\newtheorem{prop}[theorem]{Proposition}
\theoremstyle{definition}

\theoremstyle{remark}
\newtheorem{rem}[theorem]{Remark}

\numberwithin{equation}{section}
\numberwithin{theorem}{section}

 \newcommand{\norm}[1]{\left\Vert#1\right\Vert}

\newcommand{\abs}[1]{\left\vert#1\right\vert}

\newcommand{\C}{{\mathbb{C}}}

\def\be{\begin{equation}}
\def\ee{\end{equation}}

\def \e{\epsilon}

\def\ep{\epsilon}

\def\O{\mathcal O}

\begin{document}
\title[Stability of multi-solitons in NLS]{\bf Stability of multi-solitons in the cubic NLS equation}

\author[A. Contreras, D.E. Pelinovsky]{Andres Contreras and Dmitry E. Pelinovsky}

\address{Department of Mathematics and Statistics, McMaster
University, Hamilton ON, Canada, L8S 4K1}

\date{\today}
\maketitle

\begin{abstract}
We address stability of multi-solitons in the cubic NLS (nonlinear Schr\"{o}dinger) equation
on the line. By using the dressing transformation and the inverse scattering transform methods,
we obtain the orbital stability of multi-solitons in the $L^2(\mathbb{R})$ space when
the initial data is in a weighted $L^2(\mathbb{R})$ space.
\end{abstract}


\section{Introduction}

We study stability properties of multi-solitons in
the cubic NLS (nonlinear Schr\"{o}dinger) equation
\be
\label{NLS}
i \partial_t q + \partial_x^2 q + 2 \abs{q}^2q = 0,
\ee
where $q(x,t) : \mathbb{R} \times \mathbb{R} \to \mathbb{C}$.
The initial value problem for the cubic NLS equation
(\ref{NLS}) associated with initial data $q |_{t = 0} = q_0$
is locally well-posed in $L^2(\mathbb{R})$
thanks to the result of Tsutsumi based on Stritcharz inequalities \cite{tsutsumi}. It
has a global solution in $L^2(\mathbb{R})$ thanks to the conservation
of the $L^2(\mathbb{R})$ norm in time $t$:
\begin{equation}
\label{conservation}
\| q(\cdot,t) \|_{L^2(\mathbb{R})} = \| q_0 \|_{L^2(\mathbb{R})}, \quad t \in \mathbb{R}.
\end{equation}

The cubic NLS equation (\ref{NLS}) can be studied by methods of the direct and
inverse scattering transforms known since the two classical works of Zakharov and Shabat \cite{ZS-NLS1,ZS-NLS2}.
There exists a vast literature on various applications of the inverse scattering
transform methods to the cubic NLS equation (\ref{NLS}), which we do not intend to overview
(see, e.g., the recent book \cite{AblowitzSegur,Ablowitz}).

Our particular emphasis is on the problem of nonlinear stability of multi-solitons, which are
known to exist in the explicit form \cite{ZS-NLS1,ZS-NLS2}.
Spectral and orbital stability of $n$-solitons in Sobolev space $H^n(\mathbb{R})$ was
proved by Kapitula \cite{Kap}, based on the classical work of Grillakis, Shatah, and Strauss
on stability of $1$-solitons in $H^1(\mathbb{R})$ \cite{GSS}.

Functional analytic methods were developed to study interactions of many widely separated solitons
in $H^1(\mathbb{R})$ for the NLS equation with cubic and higher-order nonlinearities, in particular,
by Perelman \cite{Perelman}, Rodnianski, Schlag, and Soffer \cite{Rod}, and Martel, Merle, and Tsai \cite{Merle}.
Recent progress along this direction includes the work of Holmer and Zworski \cite{HZ} on interaction of
a soliton with a $\delta$-distribution impurity and the work of Holmer and Lin \cite{Holmer} on
interactions of two solitons. Because the inverse scattering transform methods are not used
in this literature, the results are usually weaker than those obtained with the
inverse scattering transform methods.

New results on stability of $1$-solitons were obtained recently in the context of the cubic NLS equation (\ref{NLS}) by
combining functional analytic methods and the inverse scattering transform. Deift and Park \cite{DP}
computed long-time asymptotics for the NLS equation with a delta potential
to prove asymptotic stability of soliton-defect modes in a weighted $L^2(\mathbb{R})$ space
and to improve earlier results of Holmer and Zworski \cite{HZ}.
Mizumachi and Pelinovsky \cite{MP} proved orbital stability of $1$-solitons in $L^2(\mathbb{R})$
improving the standard results of Grillakis, Shatah, and Strauss \cite{GSS}.
Cuccagna and Pelinovsky \cite{CP} proved asymptotic stability of $1$-solitons in a weighted
$L^2(\mathbb{R})$ space by combining the inverse scattering and the steepest descent method,
which was earlier developed in a different context by Deift and Zhou \cite{DZ-book,DZ,Z}.

In this paper, we would like to extend the results of \cite{CP,DP,MP} to multi-solitons
of the cubic NLS equation (\ref{NLS}) by combining the inverse scattering transform methods
and the dressing transformation, which was developed by Zakharov and Shabat long ago \cite{ZS1,ZS2} (see, e.g.,
Chapter 3 in book \cite{ZMNP}).

We denote by $L^{2,s}(\mathbb{R})$ the weighted $L^2(\mathbb{R})$
space with the norm
\begin{equation}
\| u \|_{L^{2,s}(\mathbb{R})} := \| \langle x \rangle^s u \|_{L^2(\mathbb{R})}, \quad \langle x \rangle := \sqrt{1 + x^2}.
\end{equation}
The following theorem gives the main result of this article.

\begin{theorem}
Let $q^S$ be a $n$-soliton solution of the cubic NLS equation (\ref{NLS}) with real parameters
$\{ \xi_j, \eta_j, x_j, \theta_j \}_{j = 1}^n$ such that
pairs $(\xi_j,\eta_j)$ are all distinct. There exist positive constants $\epsilon_0$
and $C$ such that if $q_0 \in L^{2,s}(\mathbb{R})$ for any $s > \frac{1}{2}$ and if
\begin{equation}
\label{closeness}
\epsilon := \| q_0 - q^S(\cdot,0) \|_{L^{2,s}(\mathbb{R})} < \epsilon_0
\end{equation}
then there exist a solution $q$ of the cubic NLS equation (\ref{NLS}) with $q |_{t = 0} = q_0$
and an $n$-soliton solution $q^{S'}$ with real parameters
$\{ {\xi}'_j, {\eta}'_j, {x}'_j, {\theta}'_j \}_{j = 1}^n$
such that
\begin{equation}
\label{bounds-1}
\max_{1 \leq j \leq n} \left| ({\xi}'_j,{\eta}'_j,{x}'_j,{\theta}'_j) -
(\xi_j, \eta_j, x_j, \theta_j) \right| \leq C \epsilon
\end{equation}
and
\begin{equation}
\label{bounds-2}
\| q(\cdot,t) - q^{S'}(\cdot,t) \|_{L^2(\mathbb{R})} \leq C \epsilon, \quad t \in \mathbb{R}.
\end{equation}
\label{theorem-main}
\end{theorem}

\begin{rem}
The orbital stability result (\ref{bounds-2}) is not expected in the $L^{2,s}(\mathbb{R})$
norm for $s > 0$ because, in a general situation, the variance of the solution $\| x q \|_{L^2(\mathbb{R})}$
grows linearly in time $t$, according to the cubic NLS equation (\ref{NLS}).
\end{rem}

For $n = 1$, the $L^2(\mathbb{R})$ orbital stability result of Theorem \ref{theorem-main} was proved by
Mizumachi and Pelinovsky \cite{MP} without the requirement that the initial data is close to the
$1$-soliton in the weighted $L^2(\mathbb{R})$ space (\ref{closeness}). Without this
requirement, the initial data can support more than one soliton in the long-time asymptotics,
but the additional solitons have small $L^2(\mathbb{R})$ norm and are included in the residual terms
of the $L^2(\mathbb{R})$ orbital stability result (\ref{bounds-2}).

At the same time, a stronger result on the asymptotic stability of $1$-solitons was proved
by Cuccagna and Pelinovsky \cite{CP} under the requirement (\ref{closeness}) with the decay
of the $L^{\infty}(\mathbb{R})$ norm of the residual term in time:
\begin{equation}
\| q(\cdot,t) - q^{S'}(\cdot,t) \|_{L^{\infty}(\mathbb{R})} \leq C \epsilon t^{-1/2}, \quad \mbox{\rm as} \quad t \to \infty.
\end{equation}

We believe that both the orbital stability if $q_0 \in L^2(\mathbb{R})$
and the asymptotic stability if $q_0 \in L^{2,s}(\mathbb{R})$ with $s > \frac{1}{2}$
hold for $n$ solitons but the proofs of these two refinements of Theorem \ref{theorem-main}
would require considerable lengthening of the present work at the possible expense of
obscuring the main argument. Note that a similar constraint on the initial data
to enable the inverse scattering transform methods is used by Gerard and Zhang \cite{GZ}
to prove the $L^{\infty}(\mathbb{R})$ orbital stability of black solitons in the cubic defocussing NLS equation.

For $n = 2$, the result of Theorem \ref{theorem-main} provides orbital stability of the class of $2$-solitons in the cubic
NLS equation (\ref{NLS}). Note that interactions of widely separated two solitons in this equation
was recently considered by Holmer and Lin \cite{Holmer} using an effective dynamical
equation that describes the solitons dynamics for large but finite time. Depending on the parameters
$(\xi_1,\xi_2,\eta_1,\eta_2)$, the two solitons collide and scatter with different velocities
but may also form a bound state (a time-oscillating space-localized breather). Although Theorem \ref{theorem-main}
provides orbital stability of the entire family of $2$-solitons, this result does not exclude
the phenomenon of instability of the $2$-soliton breather. Indeed, if the breather corresponds
to the constraint $\xi_1 = \xi_2$, whereas the initial condition yields $\xi_1' \neq \xi_2'$,
the time-oscillating breather is destroyed by the initial perturbation and transforms into
two solitons moving with different velocities.

The paper is organized as follows. Section 2 gives details of the dressing transformation,
which enables us to construct $n$-solitons from the trivial zero solution of the cubic NLS equation (\ref{NLS}).
Section 3 describes the time evolution for the dressing transformation. The exact $n$-solitons
are constructed in Section 4. Section 5 is devoted to analysis of the mapping between
the $L^2$-neighborhood of the zero solution and the $L^2$-neighborhood of the $n$-solitons:
the mapping is one-to-one but it is not onto, unless the constraints on soliton parameters
are imposed. Instead of adding constraints, we determine in Section 6 parameters of the $n$-soliton $q^{S'}$
based on the inverse scattering transform methods, which are enabled by adding the requirement (\ref{closeness})
on the initial data $q_0$. All together, these arguments will complete the proof of orbital
stability of $n$-solitons in $L^2(\mathbb{R})$.

\vspace{0.25cm}

{\bf Acknowledgements:} A.C. is supported by the postdoctoral fellowship at McMaster University.
D.P. is supported in part by the NSERC Discovery grant. The authors are indebted to S. Cuccagna 
for bringing this problem up and for critical remarks during the preparation of this manuscript.

\vspace{0.25cm}

\section{Dressing transformation}

We use the dressing method of Zakharov and Shabat \cite{ZS1,ZS2}
to map a neighborhood of the zero solution to a neighborhood of a multi-soliton
in the cubic NLS equation (\ref{NLS}). The dressing method relies on the existence
of the Lax operator pair
\begin{equation}
\label{eq:zs}
\left\{ \begin{array}{l} \partial_x \psi = - i z \sigma_3 \psi + Q(q) \psi, \\
\partial_t \psi = i (|q|^2 - 2 z^2) \sigma_3 \psi + 2 z Q(q) \psi - i Q(\partial_x q) \sigma_3 \psi,
\end{array} \right.
\end{equation}
where $\psi(x,t) : \mathbb{R} \times \mathbb{R} \to \C^2$ and
\begin{equation*}
Q(q) := \begin{pmatrix} 0  &  q \\ - \overline{q}  & 0 \end{pmatrix}, \quad
\sigma _3 :=    \begin{pmatrix} 1  & 0 \\ 0 & -1 \end{pmatrix}.
\end{equation*}
The compatibility condition $\psi_{xt} = \psi_{tx}$ for a classical
solution $\psi \in C^2(\mathbb{R} \times \mathbb{R}, \mathbb{C}^2)$
of the Lax system (\ref{eq:zs}) with constant spectral parameter $z$ is equivalent
to the requirement that $q(x,t)$ is a classical solution of the NLS equation (\ref{NLS}),
that is, $q$ is $C^1$ in $t$ and $C^2$ in $x$ for all $(x,t) \in \mathbb{R} \times \mathbb{R}$.

Let $\Phi(x,t,z)$ be a fundamental matrix solution of the system (\ref{eq:zs}) such that
$\Phi(0,t,z) = I$, where $I$ is a $2\times 2$ identity matrix.
In what follows, we will only consider the first part of the system
(\ref{eq:zs}) and will set $\phi(x,z) := \Phi(x,0,z)$. The time evolution
will be added in the next section, according to the standard analysis.

We define a fundamental matrix solution $\phi(x,z) : \mathbb{R} \times \mathbb{C} \to \mathbb{M}^{2 \times 2}$
from the system of differential equations:
\begin{equation}
\label{fund-matrix}
\left\{ \begin{array}{l} \partial_x \phi = U(x,z) \phi, \\
\phi |_{x = 0} = I, \end{array} \right. \quad U(x,z) := - i z \sigma_3 + Q(q).
\end{equation}
Since $U^+ = - U$, where $U^+ = \bar{U}^T$, the fundamental matrix
is inverted by the following elementary result.

\begin{prop}
Let $\phi$ be a fundamental matrix
solution of the system (\ref{fund-matrix}). Then, $\phi$ is invertible
and $\phi^{-1}(x,z) = \phi^+(x,z)$.
\end{prop}

\begin{proof}
We verify that
$$
\partial_x \left[  \phi^+(x,z) \phi(x,z) \right] = \phi^+(x,z) U^+(x,z) \phi(x,z) +
\phi^+(x,z) U(x,z) \phi(x,z) = 0,
$$
so that $\phi^+(x,z) \phi(x,z)$ is constant in $x$ and equals to $I$ because $\phi^+(0,z) \phi(0,z) = I$.
A similar computation holds for $\phi(x,z) \phi^+(x,z)$, hence $\phi^+(x,z)$ is the inverse for
$\phi(x,z)$.
\end{proof}

Let $\phi_0$ denote the fundamental matrix solution of the system (\ref{fund-matrix})
for the potential $q_0$. In what follows, $q_0$ is not related to the initial
data in Theorem \ref{theorem-main} but denotes another (simpler) solution of the
cubic NLS equation (\ref{NLS}). Let us define
the matrix function $\phi$  by the {\em dressing transformation} formula:
\begin{equation}
\label{dressing-trans}
\phi(x,z) := \chi(x,z) \phi_0(x,z), \quad \chi(x,z) = I + \sum_{k=1}^n \frac{i r_k(x) \otimes \bar{s}_k(x)}{z - z_k},
\end{equation}
where for each $k$, ${\rm Im}(z_k) > 0$, $r_k(x), s_k(x) : \mathbb{R} \to \mathbb{C}^2$
are to be defined, and $\otimes$ denotes an outer product of vectors in $\mathbb{C}^2$
(without complex conjugation). The factor $i$ is used in the sum for convenience.
The following result summarizes the dressing method of Zakharov and Shabat \cite{ZS1,ZS2}.
For convenience of readers, we give a precise proof of the dressing transformation.

\begin{prop}
Assume $q_0 \in C(\mathbb{R})$ and define the set $\{ s_k \}_{k=1}^n$ from
classical solutions of the Zakharov--Shabat (ZS) system
\begin{equation}
\label{solution-s}
\partial_x s_k = - i \bar{z}_k \sigma_3 s_k + Q(q_0) s_k, \quad 1 \leq k \leq n,
\end{equation}
such that the Gramian-type matrix with entries
\begin{equation}
\label{Gramian}
M_{k,j} := \frac{-i}{\bar{z}_k - z_j} \langle s_j, s_k \rangle, \quad 1 \leq k,j \leq n
\end{equation}
is invertible. Let the set $\{ r_k \}_{k=1}^n$ be defined from the set $\{ s_k \}_{k=1}^n$  by
unique solution of the linear system
\begin{equation}
\label{lin-system-1}
i s_k = \sum_{j=1}^n \frac{\langle s_j, s_k \rangle}{\bar{z}_k - z_j} r_j, \quad 1 \leq k \leq n,
\end{equation}
with the inverse
\begin{equation}
\label{lin-system-2}
i r_k = \sum_{j=1}^n \frac{\langle r_j, r_k \rangle}{\bar{z}_j - z_k} s_j, \quad 1 \leq k \leq n,
\end{equation}
where $\langle u, v \rangle := \bar{u}_1 v_1 + \bar{u}_2 v_2$ is the dot product between vectors in $\C^2$.
Then, the dressing transformation (\ref{dressing-trans}) is invertible with the inverse
\begin{equation}
\label{inverse-chi}
\phi^{-1}(x,z) = \phi_0^+(x,z) \chi^+(x,z), \quad
\chi^+(x,z) = I - \sum_{k=1}^n \frac{i s_k(x) \otimes \bar{r}_k(x)}{z - \bar{z}_k},
\end{equation}
and $\phi$ is a solution of the system $\partial_x \phi = U(x,z) \phi$
for the potential $q$, which is related to the potential $q_0$ by
the transformation formula
\begin{equation}
\label{transformation}
Q(q) = Q(q_0) + \sum_{k=1}^n r_k \otimes \bar{s}_k \sigma_3 - \sigma_3 s_k \otimes \bar{r}_k.
\end{equation}
In addition, the set $\{ r_k \}_{k = 1}^{n}$ satisfies the ZS system
\begin{equation}
\label{solution-r}
\partial_x r_k = - i z_k \sigma_3 r_k + Q(q) r_k, \quad 1 \leq k \leq n,
\end{equation}
associated with the same potential $q$.
\label{proposition-dressing}
\end{prop}

\begin{proof}
First, we check that
$$
\chi^+(x,z) \chi(x,z) =  \chi(x,z) \chi^+(x,z) = I,
$$
which yields
$$
\phi^+(x,z) \phi(x,z) = \phi(x,z) \phi^+(x,z) = I.
$$
We use the partial fraction
$$
\frac{1}{(z-\bar{z}_j)(z-z_k)} = \frac{1}{\bar{z}_j - z_k} \left[ \frac{1}{z - \bar{z}_j} - \frac{1}{z - z_k} \right].
$$
Then $\chi(x,z) \chi^+(x,z) = I$ is equivalent to the system
$$
i r_k \otimes \bar{s}_k = \sum_{j = 1}^n \frac{(r_k \otimes \bar{s}_k) \; (s_j \otimes \bar{r}_j)}{\bar{z}_j - z_k},
$$
which yields the system (\ref{lin-system-1}) after projection to $r_k$ from the left. 

Now, since $\chi(x,z)$ is a square $n \times n$ matrix and 
$\chi(x,z) \chi^+(x,z) = I$, then $|\det(\chi(x,z))| = 1$ and therefore, $\chi(x,z)$ is invertible with 
$\chi^{-1}(x,z) = \chi^+(x,z)$. On the other hand, $\chi^+(x,z) \chi(x,z) = I$ is equivalent to the system
$$
i r_k \otimes \bar{s}_k = \sum_{j = 1}^n \frac{(s_j \otimes \bar{r}_j) \; (r_k \otimes \bar{s}_k)}{\bar{z}_j - z_k},
$$
which yields the system (\ref{lin-system-2}) after projection to $s_k$ from the right.
Therefore, the system (\ref{lin-system-2}) is inverse to the system (\ref{lin-system-1})). 
Note that all vectors in the sets $\{ s_k \}_{k=1}^n$ and $\{ r_k \}_{k=1}^n$ are nonzero because the matrix $M$
in (\ref{Gramian}) is invertible. 

Next, we confirm that the set $\{ r_k \}_{k = 1}^{n}$, which is determined by
the linear system (\ref{lin-system-1}), satisfies the ZS system
(\ref{solution-r}) with the potential $q$ if
the set $\{ s_k \}_{k=1}^n$ satisfies the ZS system
(\ref{solution-s}) with the potential $q_0$, where $q$ and $q_0$ are related by
the transformation formula (\ref{transformation}). Differentiating the linear system (\ref{lin-system-1})
in $x$ and substituting (\ref{solution-s}), we obtain
\begin{eqnarray*}
\bar{z}_k \sigma_3 s_k + i Q(q_0) s_k = - i \sum_{j =1}^n \langle \sigma_3 s_j, s_k \rangle
r_j + \sum_{j =1}^n \frac{\langle s_j, s_k \rangle}{\bar{z}_k - z_j} \partial_x r_j.
\end{eqnarray*}
Using the transformation formula (\ref{transformation}) and the inverse linear system (\ref{lin-system-2}),
we rewrite this equation as follows:
\begin{eqnarray*}
\sum_{j =1}^n \frac{\langle s_j, s_k \rangle}{\bar{z}_k - z_j} \partial_x r_j & = &
\bar{z}_k \sigma_3 s_k + \sum_{j =1}^n \frac{\langle s_j, s_k \rangle}{\bar{z}_k - z_j} Q(q) r_j + i
\sigma_3 \sum_{j =1}^n \langle r_j, s_k \rangle s_j \\
& = &
\bar{z}_k \sigma_3 s_k + \sum_{j =1}^n \frac{\langle s_j, s_k \rangle}{\bar{z}_k - z_j} Q(q) r_j + i
\sigma_3 \sum_{j =1}^n \langle s_j, s_k \rangle r_j,
\end{eqnarray*}
where the following transformation was used:
\begin{eqnarray}
\label{transform-help}
i \sum_{m =1}^n \langle r_m, s_k \rangle s_m & = & \sum_{m =1}^n \sum_{j =1}^n
\frac{\langle r_m, r_j \rangle \langle s_j, s_k \rangle}{\bar{z}_m - z_j} s_m \\
& = & \sum_{j =1}^n \langle s_j, s_k \rangle \sum_{m =1}^n
\frac{\langle r_m, r_j \rangle}{\bar{z}_m-z_j} s_m \nonumber \\
& = & i \sum_{j =1}^n \langle s_j, s_k \rangle r_j. \nonumber
\end{eqnarray}
Using the linear system (\ref{lin-system-1}) again, we obtain
\begin{eqnarray*}
\sum_{j =1}^n \frac{\langle s_j, s_k \rangle}{\bar{z}_k - z_j} \left[ \partial_x r_j + i z_j \sigma_3 r_j - Q(q) r_j \right] = 0,
\end{eqnarray*}
which yields the ZS system (\ref{solution-r}), because the matrix $M$ in (\ref{Gramian}) is invertible.

We shall now verify that $\phi(x,z)$ is a solution of the system $\partial_x \phi = U \phi$ from the condition
\begin{equation}
\label{potential-relation}
U(x,z) = \partial_x \phi(x,z) \phi^{-1}(x,z) = \left[ \partial_x \chi(x,z) + \chi(x,z) U_0(x,z)\right] \chi^+(x,z),
\end{equation}
where $U_0(x,z) = \partial_x \phi_0(x,z) \phi_0^{-1}(x,z)$. By using the partial fraction decompositions,
we shall first remove the residue terms at simple poles of equation (\ref{potential-relation}).

The residue terms at $\mathcal{O}\left(\frac{1}{z-z_k}\right)$ are removed from equation (\ref{potential-relation}) if
\begin{eqnarray*}
\left[ i \partial_x \left( r_k \otimes \bar{s}_k \right) + z_k r_k \otimes \bar{s}_k \sigma_3
+ i r_k \otimes \bar{s}_k Q(q_0) \right]
 \left[ I - i  \sum_{j = 1}^n \frac{s_j \otimes \bar{r}_j}{z_k - \bar{z}_j} \right] = 0.
\end{eqnarray*}
Because of the linear system (\ref{lin-system-1}), this equation simplifies to the form
\begin{eqnarray*}
r_k \otimes \left[ i \partial_x \bar{s}_k + z_k \bar{s}_k \sigma_3
+ i \bar{s}_k Q(q_0) \right]
 \left[ I - i  \sum_{j = 1}^n \frac{s_j \otimes \bar{r}_j}{z_k - \bar{z}_j} \right] = 0.
\end{eqnarray*}
Projection to $r_k$ from the left (assuming $r_k$ is nonzero) and Hermite conjugation
with the help of equation $Q^+(q_0) = - Q(q_0)$ yields the new equation
\begin{eqnarray*}
\left[ I + i  \sum_{j = 1}^n \frac{r_j \otimes \bar{s}_j}{\bar{z}_k - z_j} \right]
\left[ -i \partial_x s_k + \bar{z}_k \sigma_3 s_k + i Q(q_0) s_k \right] = 0,
\end{eqnarray*}
which is satisfied if $s_k$ is a nonzero solution of the ZS system (\ref{solution-s}). Note that the
operator on the left has a nontrivial kernel, hence the ZS system (\ref{solution-s}) is only a
particular solution of the constraint.

Next, the residue terms at $\mathcal{O}\left(\frac{1}{z-\bar{z}_j}\right)$ are removed
from equation (\ref{potential-relation}) if
\begin{eqnarray*}
\left[ \partial_x \left( \sum_{k=1}^n \frac{r_k \otimes \bar{s}_k}{\bar{z}_j-z_k} \right)
- \bar{z}_j \sigma_3 - i Q(q_0) - i \bar{z}_j \sum_{k=1}^n \frac{r_k \otimes \bar{s}_k}{\bar{z}_j-z_k}  \sigma_3
+\sum_{k=1}^n \frac{r_k \otimes \bar{s}_k}{\bar{z}_j-z_k}  Q(q_0) \right]
s_j \otimes \bar{r}_j = 0.
\end{eqnarray*}
Projection to $r_j$ from the right yields the new equation
\begin{eqnarray*}
\left[ \partial_x \left( \sum_{k=1}^n \frac{r_k \otimes \bar{s}_k}{\bar{z}_j-z_k} \right)
- \bar{z}_j \sigma_3 - i Q(q_0) - i \bar{z}_j \sum_{k=1}^n \frac{r_k \otimes \bar{s}_k}{\bar{z}_j-z_k}  \sigma_3
+\sum_{k=1}^n \frac{r_k \otimes \bar{s}_k}{\bar{z}_j-z_k}  Q(q_0) \right]
s_j = 0.
\end{eqnarray*}
If $s_k$ is a solution of the ZS system (\ref{solution-s}), then this equation reduces further to the form
\begin{eqnarray*}
\left[ \sum_{k=1}^n \frac{(\partial_x  r_k) \otimes \bar{s}_k}{\bar{z}_j-z_k}
- \bar{z}_j \sigma_3 - i Q(q_0) - i \sum_{k=1}^n r_k \otimes \bar{s}_k  \sigma_3 \right] s_j = 0,
\end{eqnarray*}
where the derivative in $x$ applies now to $r_k$ only.
This equation is rewritten with the help of the reconstruction formula (\ref{transformation}) in the form
\begin{eqnarray*}
\sum_{k=1}^n \frac{\langle s_k, s_j \rangle}{\bar{z}_j-z_k} \partial_x  r_k
- \bar{z}_j \sigma_3 s_j - i Q(q) s_j - i \sigma_3 \sum_{m =1}^n \langle r_m, s_j \rangle s_m = 0.
\end{eqnarray*}
Substituting transformation (\ref{transform-help}) to the previous equation
and using the linear system (\ref{lin-system-1}), we derive
\begin{eqnarray*}
\sum_{k=1}^n \frac{\langle s_k, s_j \rangle}{\bar{z}_j-z_k} \left[ \partial_x  r_k
- Q(q) r_k + i z_k \sigma_3 r_k \right]
- \bar{z}_j \sigma_3 s_j - i \bar{z}_j \sigma_3 \sum_{k =1}^n \frac{\langle s_k, s_j \rangle}{\bar{z}_j - z_k} r_k = 0.
\end{eqnarray*}
The last two terms cancel out thanks to the linear system (\ref{lin-system-1}). As a result,
the equation is satisfied if $r_k$ is a solution of the ZS system (\ref{solution-r}).

Now, since all the residue terms are removed from equation (\ref{potential-relation}),
this equation reduces a single equation, which is nothing but
the reconstruction formula (\ref{transformation}).
\end{proof}

\begin{rem}
Note that $\phi$ given by (\ref{dressing-trans}) is not the fundamental matrix solution
of the system (\ref{fund-matrix}) because
$\phi|_{x = 0} = I$ is not satisfied. However, if we define
$$
\phi(x,z) := \chi^+(0,z) \chi(x,z) \phi_0(x,z),
$$
then this $\phi$ is the fundamental
matrix solution of the system (\ref{fund-matrix}).
\end{rem}

For $1$-soliton solutions with $n = 1$, the result of Proposition \ref{proposition-dressing} can
be simplified as follows. Let $q_0 = 0$, $z_1 = \xi_1 + i \eta_1$ with $\eta_1 > 0$, and
let $s_1 = (\mathfrak{b}_1,\mathfrak{b}_2)$ be a solution of the ZS system (\ref{solution-s})
with $q_0 = 0$, or explicitly:
\begin{equation}
\label{system-b}
\left\{ \begin{array}{l} \partial_x \mathfrak{b}_1 = -(\eta_1 + i \xi_1) \mathfrak{b}_1, \\
\partial_x \mathfrak{b}_2 = (\eta_1 + i \xi_1) \mathfrak{b}_2. \end{array} \right.
\end{equation}
Then, $r_1 = (\mathfrak{a}_1,\mathfrak{a}_2)$ is found from the linear system (\ref{lin-system-1}) in the closed form:
\begin{equation}
\label{1-soliton-system}
\left[ \begin{array}{c} \mathfrak{a}_1 \\ \mathfrak{a}_2 \end{array} \right] =
\frac{2\eta_1}{|\mathfrak{b}_1|^2 + |\mathfrak{b}_2|^2} \left[ \begin{array}{c}\mathfrak{b}_1 \\ \mathfrak{b}_2 \end{array} \right].
\end{equation}
Note that $r_1$ is an eigenfunction for an isolated eigenvalue $z_1 = \xi_1 + i \eta_1$ of
the ZS system (\ref{solution-r}) associated with the $1$-soliton. The transformation formula
(\ref{transformation}) yields $1$-soliton:
\begin{equation}
\label{1-soliton-reconstruction}
q = - \mathfrak{a}_1 \bar{\mathfrak{b}}_2 - \bar{\mathfrak{a}}_2 \mathfrak{b}_1
= \frac{-4 \eta_1 \mathfrak{b}_1 \bar{\mathfrak{b}}_2}{|\mathfrak{b}_1|^2 + |\mathfrak{b}_2|^2}.
\end{equation}
Setting a general solution of the ZS system (\ref{system-b}) in the form
$$
\left\{ \begin{array}{c}
\mathfrak{b}_1 = e^{-(\eta_1 + i \xi_1)(x-x_0) + i \theta}, \\
\mathfrak{b}_2 = - e^{(\eta_1 + i \xi_1)(x-x_0) - i \theta}, \end{array} \right.
$$
where $x_0, \theta \in \mathbb{R}$ are arbitrary parameters, we obtain
from (\ref{1-soliton-system}) and (\ref{1-soliton-reconstruction})
the $1$-soliton with four arbitrary parameters:
\begin{equation}
\label{1-soliton}
q = 2 \eta_1 {\rm sech}(2 \eta_1 (x-x_0)) e^{2 i \theta - 2 i \xi_1 (x-x_0)}.
\end{equation}
Note that parameters $(x_0,\theta)$ can be set to zero by using the translational
and gauge transformations of the cubic NLS equation (\ref{NLS}),
parameter $\xi_1$ can be set to zero by using the Galileo transformation,
and parameter $\eta_1$ can be fixed at any positive number because of the scaling transformation.

\begin{rem}
Note that the result of the dressing method is different from
the result of the auto-Backlund transformation used in recent papers \cite{CP,DP,MP},
where the vector $(\mathfrak{b}_1,\mathfrak{b}_2)$ in the soliton reconstruction formula
(\ref{1-soliton-reconstruction}) was defined in terms of the solution of the spectral system (\ref{solution-s})
with $z_1$ instead of $\bar{z}_1$.
\end{rem}

\section{Time evolution of the dressing transformation}

Before looking at the time evolution of the dressing transformations, let us give an explicit
representation of the dressing transformation with the help of matrix algebra.

Let the set $\{ s_k \}_{k=1}^n$ be defined by the solutions of
the Zakharov--Shabat system (\ref{solution-s}) such that the Gramian-type matrix
$M$ in (\ref{Gramian}) is invertible. Vectors $\{ r_k \}_{k = 1}^n$
are uniquely defined by the linear system (\ref{lin-system-1}).

Let $D := \det(M)$ and $D_{k,j}$ be the co-factor of the element $M_{k,j}$.
Because $M$ is invertible, we have $D \neq 0$. A unique solution of the
linear system (\ref{lin-system-1}) can be expressed in the explicit form
\begin{equation}\label{ma3}
r_k = \sum_{j=1}^n \frac{D_{j,k}}{D} s_j.
\end{equation}
Note that matrix $M$ is Hermitian and hence, $D_{j,k} = \bar{D}_{k,j}$.
Under this constraint, the transformation formula (\ref{transformation}) yields
$$
Q(q) - Q(q_0) = \sum_{k=1}^n r_k \otimes \bar{s}_k \sigma_3 - \sigma_3 s_k \otimes \bar{r}_k
= \sum_{k=1}^n \left( \begin{array}{cc} r_{k,1} \bar{s}_{k,1} - \bar{r}_{k,1} s_{k,1} &
-r_{k,1} \bar{s}_{k,2} - \bar{r}_{k,2} s_{k,1} \\ r_{k,2} \bar{s}_{k,1} + \bar{r}_{k,1} s_{k,2} &
-r_{k,2} \bar{s}_{k,2} + \bar{r}_{k,2} s_{k,2} \end{array} \right),
$$
where the diagonal entries are zeros and the off-diagonal entries yield the transformation formula
\begin{equation}
\label{transformation-q}
q - q_0 = - \sum_{k=1}^n \sum_{j=1}^n \left( \frac{D_{j,k}}{D} s_{j,1}
\bar{s}_{k,2} + \frac{\bar{D}_{j,k}}{D} s_{k,1}
\bar{s}_{j,2} \right) = -\frac{2}{D} \sum_{k=1}^n \sum_{j=1}^n  D_{j,k} s_{j,1} \bar{s}_{k,2}.
\end{equation}

The time-dependent part of the Lax operator (\ref{eq:zs}) can be included
into consideration thanks to the compatibility between the two linear equations
and the  independence of the spectral parameter $z$ from variables $(x,t)$,
under the condition that $q$ is a classical solution of the cubic NLS equation (\ref{NLS}).
Therefore, we consider the time-dependent system
\begin{equation}
\label{fund-matrix-time}
\partial_t \phi = V(x,t,z) \phi, \quad V(x,t,z) := i (|q|^2 - 2 z^2) \sigma_3  + 2 z Q(q) - i Q(\partial_x q) \sigma_3.
\end{equation}
We assume that $q_0$ is a classical solution of the cubic NLS equation (\ref{NLS})
and $\phi_0$ is a matrix solution of the system (\ref{fund-matrix-time}) for the potential $q_0$.
Then, we define the matrix function $\phi$ by the same dressing transformation formula (\ref{dressing-trans}).
The following result gives a time-dependent analogue of Proposition \ref{proposition-dressing}.

\begin{prop}
In addition to conditions of Proposition \ref{proposition-dressing},
assume that $q_0$ is a classical solution of the cubic NLS equation (\ref{NLS})
and the set $\{ s_k \}_{k = 1}^n$ yields a classical solution of
the time-evolution part of the Lax operator pair
\begin{equation}
\label{solution-s-time}
\partial_t s_k = i (|q_0|^2 - 2 \bar{z}_k^2) \sigma_3 s_k + 2 \bar{z}_k Q(q_0) s_k
- i Q(\partial_x q_0) \sigma_3 s_k, \quad 1 \leq k \leq n.
\end{equation}
Let $\phi$ be defined by the dressing transformation (\ref{dressing-trans}),
the set $\{ r_k \}_{k = 1}^n$ be defined by the linear system (\ref{lin-system-1}),
and $q$ be defined by the transformation formula (\ref{transformation}). Then,
$\phi$ is a solution of the system (\ref{fund-matrix-time}) for the potential $q$,
$\{ r_k \}_{k = 1}^n$ is a solution of the time-evolution part of the Lax operator pair
\begin{equation}
\label{solution-r-time}
\partial_t r_k = i (|q|^2 - 2 z_k^2) \sigma_3 r_k + 2 z_k Q(q) r_k
- i Q(\partial_x q) \sigma_3 r_k, \quad 1 \leq k \leq n,
\end{equation}
and $q$ is a classical solution of the cubic NLS equation (\ref{NLS}). In addition
to the transformation formula (\ref{transformation}), $q$ and $q_0$ are related by
\begin{equation}
\label{transformation-modulus}
|q|^2 = |q_0|^2 + \partial_x^2 \log(D),
\end{equation}
where $D := \det(M)$ and $M$ is given by the Gramian-type matrix (\ref{Gramian}).
\label{proposition-time}
\end{prop}

\begin{proof}
We shall prove that the linear system (\ref{lin-system-1}) and equations (\ref{solution-s-time})
yield equations (\ref{solution-r-time}). The transformation formula (\ref{transformation-modulus})
will be discovered naturally in this reduction. The fact that $\phi$ is a solution of the
system (\ref{fund-matrix-time}) for the potential $q$ follows from this reduction and
is proved similarly to Proposition \ref{proposition-dressing}. Finally, $q$ is sufficiently smooth as it is defined
by the transformation formula (\ref{transformation}). As $q$ is a compatibility condition between
systems (\ref{solution-r}) and (\ref{solution-r-time}), $q$ becomes a classical solution of the cubic NLS
equation (\ref{NLS}).

To derive equations (\ref{solution-r-time}),  we differentiate the linear system (\ref{lin-system-1}) and
substitute equations (\ref{solution-s-time}) to obtain
\begin{eqnarray*}
(2 \bar{z}_k^2-|q_0|^2) \sigma_3 s_k + 2 i \bar{z}_k Q(q_0) s_k
+ i Q(\partial_x q_0) \sigma_3 s_k \\
= -2 i \sum_{j =1}^n (\bar{z}_k + z_j)
\langle \sigma_3 s_j, s_k \rangle r_j -2 \sum_{j =1}^n
\langle Q(q_0) s_j, s_k \rangle r_j
+ \sum_{j =1}^n \frac{\langle s_j, s_k \rangle}{\bar{z}_k - z_j} \partial_t r_j.
\end{eqnarray*}
This can be written as
\begin{eqnarray*}
\sum_{j =1}^n \frac{\langle s_j, s_k \rangle}{\bar{z}_k - z_j} \left[
\partial_t r_j -  i (|q|^2 - 2 z_k^2) \sigma_3 r_k - 2 z_k Q(q) r_k
+ i Q(\partial_x q) \sigma_3 r_k \right] = G_k,
\end{eqnarray*}
where
\begin{eqnarray*}
G_k & := & (2 \bar{z}_k^2-|q_0|^2) \sigma_3 s_k + 2 i \bar{z}_k Q(q_0) s_k
+ i Q(\partial_x q_0) \sigma_3 s_k  + 2 i \sum_{j =1}^n (\bar{z}_k + z_j)
\langle \sigma_3 s_j, s_k \rangle r_j  \\
& \phantom{t} & + 2 \sum_{j =1}^n
\langle Q(q_0) s_j, s_k \rangle r_j + i \sum_{j =1}^n \frac{\langle s_j, s_k \rangle}{\bar{z}_k - z_j} (|q|^2 - 2 z_j^2) \sigma_3 r_j \\
& \phantom{t} & + 2 \sum_{j =1}^n \frac{\langle s_j, s_k \rangle}{\bar{z}_k - z_j} z_j Q(q) r_j
- i \sum_{j =1}^n \frac{\langle s_j, s_k \rangle}{\bar{z}_k - z_j} Q(\partial_x q) \sigma_3 r_j.
\end{eqnarray*}
If $G_k = 0$, then invertibility of the matrix $M$ in (\ref{Gramian}) implies validity of equations
(\ref{solution-r-time}). To show that $G_k = 0$, we use the linear system (\ref{lin-system-1}) and rewrite
\begin{eqnarray*}
G_k & = & (|q|^2 -|q_0|^2) \sigma_3 s_k + 2 i \bar{z}_k \left[ Q(q_0) - Q(q) \right] s_k
+ \left[ Q(\partial_x q_0) - Q(\partial_x q) \right] \sigma_3 s_k  \\
& \phantom{t} & + 2 i \sum_{j =1}^n (\bar{z}_k + z_j)
\langle \sigma_3 s_j, s_k \rangle r_j + 2 \sum_{j =1}^n \langle Q(q_0) s_j, s_k \rangle r_j \\
& \phantom{t} & - 2 i \sum_{j =1}^n (\bar{z}_k + z_j)
\langle s_j, s_k \rangle \sigma_3 r_j + 2 \sum_{j =1}^n \langle s_j, s_k \rangle Q(q) r_j.
\end{eqnarray*}
We can now use the transformation formula (\ref{transformation}) and its derivative version
in the following form:
$$
Q(\partial_x q) - Q(\partial_x q_0) = 2 \sum_{j=1}^n (\partial_x r_j) \otimes \bar{s}_j \sigma_3
+ r_j \otimes (\partial_x \bar{s}_j) \sigma_3 - \partial_x R,
$$
where
$$
R :=  \sum_{j=1}^n r_j \otimes \bar{s}_j \sigma_3 + \sigma_3 s_j \otimes \bar{r}_j =
\sum_{j=1}^n \left( \begin{array}{cc} r_{j,1} \bar{s}_{j,1} + \bar{r}_{j,1} s_{j,1} &
-r_{j,1} \bar{s}_{j,2} + \bar{r}_{j,2} s_{j,1} \\ r_{j,2} \bar{s}_{j,1} - \bar{r}_{j,1} s_{j,2} &
-r_{j,2} \bar{s}_{j,2} - \bar{r}_{j,2} s_{j,2} \end{array} \right).
$$
Substituting the Zakharov--Shabat systems (\ref{solution-s}) and (\ref{solution-r}), the transformation formula (\ref{transformation})
and its derivative version to the expression for $G_k$ gives
\begin{eqnarray*}
G_k = (|q|^2 -|q_0|^2 + \partial_x R ) \sigma_3 s_k.
\end{eqnarray*}
It follows from the explicit expression (\ref{ma3})
for the set $\{ r_k \}_{k=1}^n$ that $R$ is a diagonal matrix. Moreover,
the difference between the two diagonal entries of the matrix $R$
is constant in $x$ because
\begin{eqnarray*}
\sum_{k=1}^n \left( r_{k,1} \bar{s}_{k,1} + \bar{r}_{k,1} s_{k,1} + r_{k,2} \bar{s}_{k,2}
+ \bar{r}_{k,2} s_{k,2} \right) & = & \frac{2}{D} \sum_{k=1}^n \sum_{j=1}^n
D_{j,k} \langle s_k, s_1 \rangle \\
& = & \frac{2 i}{D} \sum_{k=1}^n \sum_{j=1}^n
(\bar{z}_j - z_k) D_{j,k} M_{j,k} \\
& = & 2 i \left( \sum_{j=1}^n \bar{z}_j - \sum_{k=1}^n  z_k \right).
\end{eqnarray*}
Consequently, we have
\begin{eqnarray*}
\partial_x \sum_{k=1}^n \left( r_{k,1} \bar{s}_{k,1} + \bar{r}_{k,1} s_{k,1} + r_{k,2} \bar{s}_{k,2}
+ \bar{r}_{k,2} s_{k,2} \right)  = 0.
\end{eqnarray*}
On the other hand, we have
\begin{eqnarray*}
\sum_{k=1}^n \left( r_{k,1} \bar{s}_{k,1} + \bar{r}_{k,1} s_{k,1} - r_{k,2} \bar{s}_{k,2}
- \bar{r}_{k,2} s_{k,2} \right) & = & \frac{2}{D} \sum_{k=1}^n \sum_{j=1}^n
D_{j,k} \langle \sigma_3 s_k, s_j \rangle \\
& = & -\frac{2}{D} \sum_{k=1}^n \sum_{j=1}^n
D_{j,k} \partial_x M_{j,k} \\
& = & -2 \partial_x \log(D).
\end{eqnarray*}
The last two expressions show that $G_k = 0$ if and only if $|q|^2$ and $|q_0|^2$ are related by
the transformation formula (\ref{transformation-modulus}).
\end{proof}

\begin{rem}
Transformations formulas (\ref{transformation-q}) and (\ref{transformation-modulus}) are known
to be compatible from the Hirota bilinear method \cite{AblowitzSegur} of constructing explicit solutions
of the cubic NLS equation (\ref{NLS}).
\end{rem}

By using the system (\ref{solution-s-time}) for $q_0 = 0$, $n = 1$, and $z_1 = \xi_1 + i \eta_1$ with $\eta_1 > 0$,
we include the time evolution of the $1$-soliton of the cubic NLS equation.
In this way, we find the solution for the vector $s_1 = (\mathfrak{b}_1,\mathfrak{b}_2)$ in the form
$$
\left\{ \begin{array}{l}
\mathfrak{b}_1 = e^{-(\eta_1 + i \xi_1)(x-x_0) + 2 i (\eta_1 + i \xi_1)^2 t + i \theta}, \\
\mathfrak{b}_2 = - e^{(\eta_1 + i \xi_1)(x-x_0) - 2 i (\eta_1 + i \xi_1)^2 t - i \theta},
\end{array} \right.
$$
and obtain the $1$-soliton in the form
\begin{equation}
\label{1-soliton-time}
q = 2 \eta_1 {\rm sech}(2 \eta_1 (x + 2 \xi_1 t -x_0)) e^{2 i \theta - 2 i \xi_1 (x + 2 \xi_1 t -x_0) + 4 i \eta_1^2 t}.
\end{equation}
For $t = 0$, this expression coincides with  (\ref{1-soliton}).

\section{Construction of multi-solitons}

The multi-soliton solutions of the cubic NLS equation (\ref{NLS})
are obtained by applying the dressing transformation of
Propositions \ref{proposition-dressing} and \ref{proposition-time}
with zero solution $q_0 = 0$ for general $n \geq 1$. In this case, we define
solutions of the linear systems (\ref{solution-s}) and (\ref{solution-s-time}) with
$q_0 = 0$ by
\begin{equation}\label{ma1}
s_k = e^{-i \bar{z}_k (x-x_k) \sigma_3 - 2 i \bar{z}_k^2 t \sigma_3 + i \theta_k \sigma_3}
\sigma_3 {\bf 1}, \quad 1 \leq k \leq n,
\end{equation}
where ${\bf 1} = (1,1)$, $z_k = \xi_k + i \eta_k$, and real parameters $(\xi_k, \eta_k, x_k, \theta_k)$
are arbitrary with $\eta_k > 0$. The set $\{ r_k \}_{k=1}^n$ is uniquely found
in the form (\ref{ma3}). It follows from the transformation formula (\ref{transformation-q}) with
$q_0 = 0$ that the $n$-soliton solutions are defined in the form $q = \frac{2 \Sigma}{D}$, where
\begin{equation}
\label{matrix-Sigma}
\Sigma := - \sum_{k=1}^n \sum_{j=1}^n D_{j,k} F_{j,k}
\end{equation}
and we have denoted
\begin{equation}
\label{matrix-F}
F_{j,k} := s_{j,1} \bar{s}_{k,2}, \quad 1 \leq k,j \leq n.
\end{equation}

For $2$-solitons with $n = 2$, we obtain the explicit form of the solution:
\begin{equation}\label{2sigmadelta}
q^S(x,t;\eta_1,\eta_2,\xi_1,\xi_2,x_1,x_2,\theta_1,\theta_2) = \frac{2 \Sigma^S}{D^S}
\end{equation}
with
\begin{eqnarray*}
\Sigma^S
& := & \frac{e^{-2\eta_2 \varphi_2 + 2 i \psi_1}+e^{2\eta_2 \varphi_2 + 2 i \psi_1}}{2 \eta_2}
+\frac{e^{-2\eta_1 \varphi_1+ 2 i \psi_2} + e^{2\eta_1 \varphi_1 + 2i \psi_2}}{2\eta_1} \\
& \phantom{t} &
- \frac{e^{-2\eta_2 \varphi_2 + 2 i \psi_1} + e^{2\eta_1 \varphi_1 + 2 i \psi_2}}{\eta_1+\eta_2+ i (\xi_1-\xi_2)}
- \frac{e^{-2\eta_1 \varphi_1 + 2 i \psi_2}+ e^{2\eta_2\varphi_2 + 2i \psi_1}}{\eta_1+\eta_2 - i (\xi_1-\xi_2)}
\end{eqnarray*}
and
\begin{eqnarray*}
D^S & := & \frac{(e^{-2\eta_1 \varphi_1}+e^{2\eta_1 \varphi_1})(e^{-2\eta_2\varphi_2}+e^{2\eta_2 \varphi_2})}{4\eta_1\eta_2} \\
& \phantom{t} &  -\frac{
(e^{-\eta_1 \varphi_1 - \eta_2 \varphi_2 + i \psi_1 - i \psi_2} +
e^{\eta_1 \varphi_1 + \eta_2 \varphi_2 - i \psi_1 + i \psi_2})
(e^{-\eta_1 \varphi_1 - \eta_2 \varphi_2 - i \psi_1 + i \psi_2} +
e^{\eta_1 \varphi_1 + \eta_2 \varphi_2 + i \psi_1 - i \psi_2})}{(\eta_1+\eta_2)^2+(\xi_1 - \xi_2)^2} ,
\end{eqnarray*}
where
\begin{eqnarray}
\left\{ \begin{array}{l}
\varphi_1 = x + 4 \xi_1 t - x_1, \\
\varphi_2 = x + 4 \xi_2 t - x_2, \\
\psi_1 = \theta_1 - \xi_1 (x + 2 \xi_1 t -x_1) + 2 \eta_1^2 t, \\
\psi_2 = \theta_2 - \xi_2 (x + 2 \xi_2 t -x_2) + 2 \eta_2^2 t.
\end{array} \right.
\label{definition-phases}
\end{eqnarray}

Figure \ref{fig-2-sol} shows two particular types of the dynamics of $2$-solitons:
scattering of two solitons with nonequal speeds for $\xi_1 \neq \xi_2$ (left) and oscillations of bound
states of two solitons with equal speeds for $\xi_1 = \xi_2$ (right) if $\eta_1 \neq \eta_2$. Note that the solution
becomes zero if $\xi_1 = \xi_2$ and $\eta_1 = \eta_2$.

\begin{figure}[h]
\begin{center}
\includegraphics[scale=0.45]{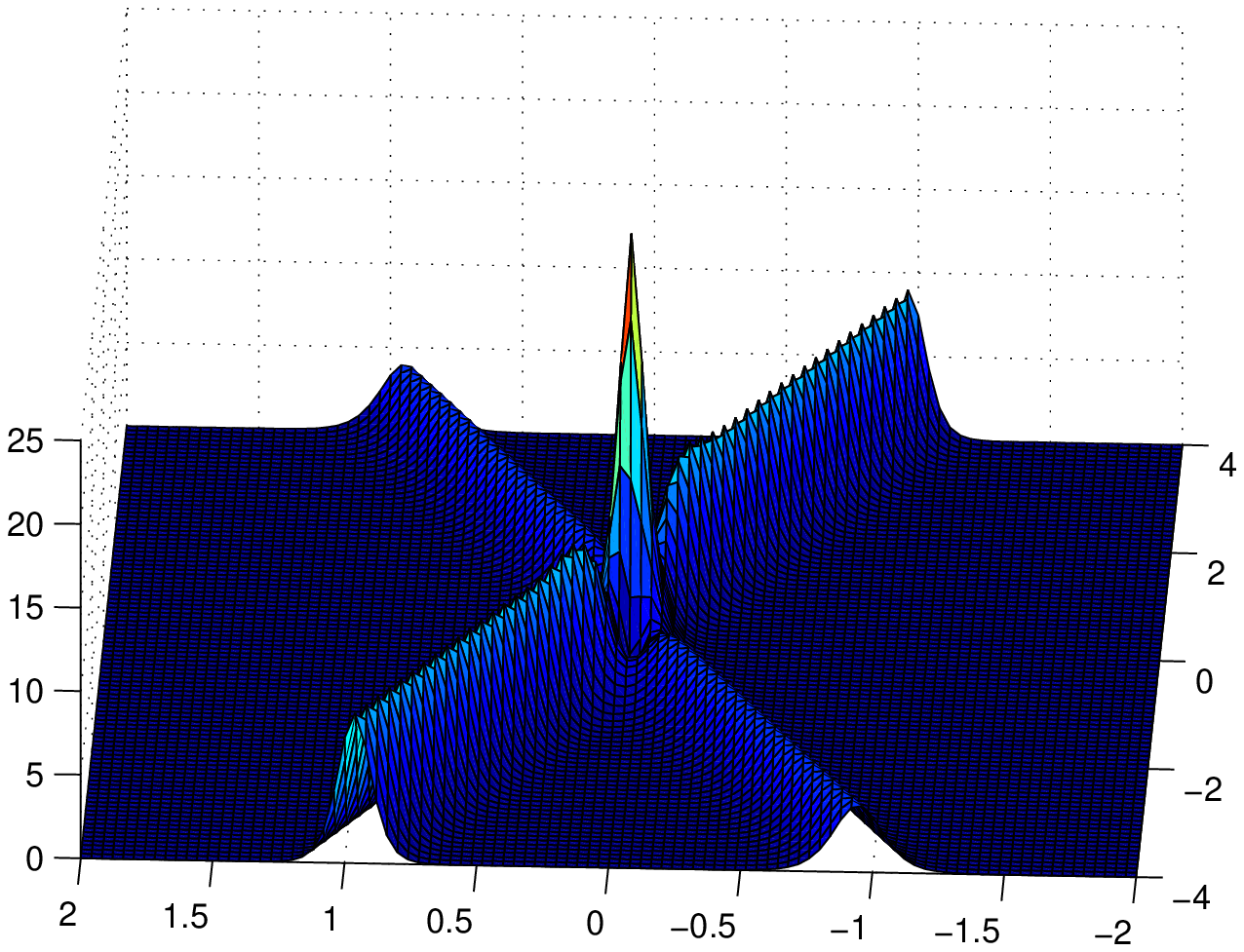}
\includegraphics[scale=0.45]{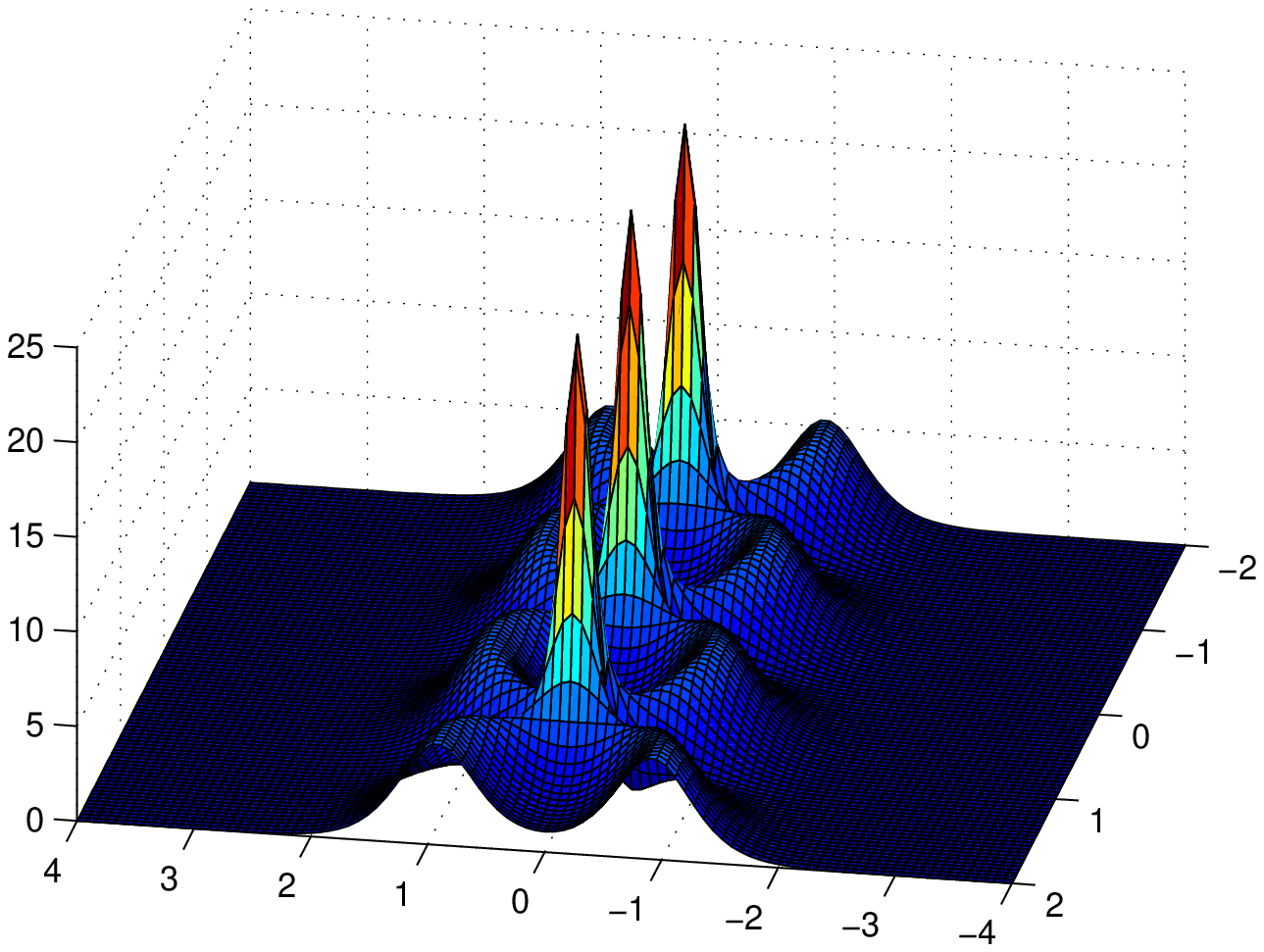}
\end{center}
\caption{\label{fig-2-sol}
Surface plots of $|q^S|^2$ versus $(x,t)$ for the $2$-soliton solutions (\ref{2sigmadelta})
with $\eta_1 = 1$, $\eta_2 = 1.5$, and either $\xi_1 = -\xi_2 = 1$ (left) or
$\xi_1 = \xi_2 = 0$ (right). Other translation parameters are set to zero.}
\end{figure}

\section{Analysis of neighborhood of multi-solitons}

We shall now consider the dressing transformation for small but nonzero $q_0$ in $L^2(\mathbb{R})$.
Recall that the $L^2(\mathbb{R})$ norm of a solution of the cubic NLS equation (\ref{NLS})
is conserved in time $t$. To work with the dressing transformation, we consider a classical solution $q_0$
of the cubic NLS equation (\ref{NLS}) in $L^2(\mathbb{R})$ with constant $\| q_0 \|_{L^2(\mathbb{R})}$.

Let $s_j$ be solutions of the spectral problems
(\ref{solution-s}) and (\ref{solution-s-time})
associated with small but nonzero $q_0$ in $L^2(\mathbb{R})$
for $z_j = \xi_j + i \eta_j$ with $\eta_j > 0$. We write this vector in the separable form
\begin{equation}
\label{defs1}
s_j := e^{-i \bar{z}_j (x-x_j) - 2 i \bar{z}_j^2 t + i \theta_j} f_j
- e^{i \bar{z}_j (x-x_j) + 2 i \bar{z}_j^2 t - i \theta_j} g_j
\end{equation}
where $(x_j,\theta_j)$ are arbitrary real parameters and
components of the $2$-vectors $f_j = (\mathfrak{a}_1^j,\mathfrak{a}_2^j)$ and
$g_j = (\mathfrak{b}_1^j,\mathfrak{b}_2^j)$ satisfy the boundary conditions
\begin{equation}
\label{defs-boundary-conditions-1}
\left\{ \begin{array}{l}
\lim\limits_{x \to -\infty} \mathfrak{a}_1^j = 1, \\
\lim\limits_{x \to +\infty} e^{-2 i \bar{z}_j (x-x_j) - 4 i \bar{z}_j^2 t + 2 i \theta_j} \mathfrak{a}_2^j = 0,
\end{array} \right.
\end{equation}
and
\begin{equation}
\label{defs-boundary-conditions-2}
\left\{ \begin{array}{l}
\lim\limits_{x \to -\infty} e^{2 i \bar{z}_j (x-x_j) + 4 i \bar{z}_j^2 t - 2 i \theta_j} \mathfrak{b}_1^j = 0, \\
\lim\limits_{x \to +\infty} \mathfrak{b}_2^j = 1. \end{array} \right.
\end{equation}
If $q_0 = 0$, then we have unique solutions $f_j = (1,0)$ and $g_j = (0,1)$, so
that the separable form (\ref{defs1}) recovers (\ref{ma1}).
Using the same analysis as in Lemmas 4.1 and 4.3 of \cite{MP}, we obtain the following result.

\begin{prop}
Let $q_0$ be a classical solution of the cubic NLS equation
(\ref{NLS}) in $L^2(\mathbb{R})$. There exists a positive constant $\e_0$ such that
if $\| q_0  \|_{L^2} \leq \e_0$, then the spectral problems
(\ref{solution-s}) and (\ref{solution-s-time}) for $z_j = \xi_j + i \eta_j$
admit a solution $s_j$ satisfying (\ref{defs1}), (\ref{defs-boundary-conditions-1}) and (\ref{defs-boundary-conditions-2}).
For all $t \in \mathbb{R}$, components of $f_j = (\mathfrak{a}_1^j,\mathfrak{a}_2^j)$ and
$g_j = (\mathfrak{b}_1^j,\mathfrak{b}_2^j)$ belong to the class
\begin{equation}
(\mathfrak{a}_1^j,\mathfrak{a}_2^j) \in L^{\infty}(\mathbb{R}) \times (L^{\infty}(\mathbb{R}) \cap L^2(\mathbb{R})), \quad
(\mathfrak{b}_1^j,\mathfrak{b}_2^j) \in (L^{\infty}(\mathbb{R}) \cap L^2(\mathbb{R})) \times L^{\infty}(\mathbb{R}),
\end{equation}
and there exists a positive $q_0$-independent constant $C$ such that
\begin{eqnarray}
\label{fghl}
\norm{\mathfrak{a}_1^j-1}_{L^{\infty}}+\norm{\mathfrak{a}_2^j}_{L^2\cap L^{\infty}}
+ \norm{\mathfrak{b}_1^j}_{L^2\cap L^{\infty}} + \norm{\mathfrak{b}_2^j - 1}_{L^{\infty}} \leq C
\norm{q_0}_{L^2}.
\end{eqnarray}
\end{prop}

We now construct a neighborhood of a multi-soliton by using the dressing
transformation in Propositions \ref{proposition-dressing} and \ref{proposition-time}.
The arguments are valid for all multi-solitons, but we give details of
analysis in the case of $2$-solitons, because of the nature of the perturbations
yielding page-long computations. At the end of the section, we summarize
the key steps and modifications required to obtain the result in the general case $n \geq 1$.

Let $M$, $\Sigma$, and $F$ be the matrices defined by
\eqref{Gramian}, \eqref{matrix-Sigma}, and \eqref{matrix-F},
associated to $s_1$ and $s_2$ as given in  \eqref{defs1}.
Let $M^S$, $\Sigma^S$, and $F^S$ be the matrices corresponding  to the 2-soliton $q^S$
given by \eqref{2sigmadelta}. We denote
$D = \det(M)$ and $D^S = \det(M^S)$.
The following result tells us that if $q_0\in L^2$ is small,
then the transformation formula \eqref{transformation-q}
with $s_1$ and $s_2$ given by \eqref{defs1} yields a new solution $q$
near the $2$-soliton $q^S$ in the $L^2$ norm.

\begin{prop}\label{PPP}
Let $q_0$ be a classical solution of the cubic NLS equation
(\ref{NLS}) in $L^2(\mathbb{R})$ and $q^S$ be the $2$-soliton given explicitly by (\ref{2sigmadelta})
with $(\xi_1,\eta_1) \neq (\xi_2,\eta_2)$.
There exists a positive constant $\e_0$ such that
if $\| q_0 \|_{L^2} \leq \e_0$, then there exists a positive $q_0$-independent constant $C$ such that
the function $q$ given by the transformation formula \eqref{transformation-q}
with the functions $s_1$ and $s_2$ in (\ref{defs1}), (\ref{defs-boundary-conditions-1}) and (\ref{defs-boundary-conditions-2})
satisfies
\be\label{normQ-Q_S}
\norm{q-q^S}_{L^{2}} \leq C \| q_0 \|_{L^2}, \quad t \in \mathbb{R}.
\ee
\end{prop}

\begin{proof}
For convenience, we express bound (\ref{normQ-Q_S}) as $\norm{q-q^S}_{L^{2}} \lesssim \ep$
with $\ep := \norm{q_0}_{L^2}$ and use these notations in
the rest of the paper. The bound (\ref{normQ-Q_S}) follows from the triangle inequality
if we can show that
\begin{equation}
\label{bound-tech}
\norm{\frac{2\Sigma}{D}-q^S}_{L^2}\lesssim\e.
\end{equation}
In turn, since in this case  $D_{i,j}=(-1)^{i+j}M_{3-i,3-j},$ this bound
will be a consequence of the bounds
\be\label{TypeI}
\norm{\frac{F_{1,1}M_{2,2}}{D}-\frac{F^S_{1,1}M^S_{2,2}}{D^S}}_{L^2},\quad
\norm{\frac{F_{2,2}M_{1,1}}{D}-\frac{F^S_{2,2}M^S_{1,1}}{D^S}}_{L^2}\lesssim\e
\ee
and
\be\label{TypeII}
\norm{\frac{F_{2,1}M_{1,2}}{D}-\frac{F^S_{2,1}M^S_{1,2}}{D^S}}_{L^2},\quad
\norm{\frac{F_{1,2}M_{2,1}}{D}-\frac{F^S_{1,2}M^S_{2,1}}{D^S}}_{L^2}\lesssim\e.
\ee

In fact, to prove \eqref{TypeI} and \eqref{TypeII}, it will suffice to show
the estimates for $F_{1,1}M_{2,2}$ and $F_{2,1}M_{1,2}$, since the estimates
for $F_{2,2}M_{1,1}$ and $F_{1,2}M_{2,1}$ are analogous.

We divide the proof into three steps.
In the first step, we write down global estimates in $L^2(\mathbb{R})$ and $L^{\infty}(\mathbb{R})$
measuring the discrepancy between $(F,M)$ and $(F^S,M^S).$
From these, in the second step, we obtain for any $t \in \mathbb{R}$ that
$$
\norm{\frac{2\Sigma}{D}-\frac{2\Sigma^S}{D^S}}_{L^{\infty}([-R,R])}\lesssim\ep,
$$
for some $R$ large but fixed.

Finally, the growth properties of $D$ (and $D^S$) together with the $L^2(\mathbb{R})$ control
on the difference between $(F,M)$ and $(F^S, M^S)$ obtained in the first part,
allows us to derive  the result outside a compact set. This result together with
the $L^{\infty}([-R,R])$ estimate from the second step yields the desired conclusion.
\vskip.4in

\underline{Step 1: Global estimates for $M_{1,2}, M_{2,2}, F_{2,1}$ and $F_{1,1}$ }\\
From \eqref{defs1}, we have
\begin{eqnarray*}
M_{1,2}=\frac{1}{\eta_1+\eta_2+i(\xi_1-\xi_2)}\Big[
e^{-\eta_2\varphi_2-i\psi_2-\eta_1\varphi_1+i\psi_1}(\mathfrak{a}_1^2\mathfrak{a}_1^1+\mathfrak{a}_2^2\mathfrak{a}_2^1)
-e^{\eta_2\varphi_2+i\psi_2-\eta_1\varphi_1+i\psi_1} \\
\times (\mathfrak{b}_1^2\mathfrak{a}_1^1+\mathfrak{b}_2^2\mathfrak{a}_2^1)
 -e^{-\eta_2\varphi_2-i\psi_2+\eta_1\varphi_1-i\psi_1}(\mathfrak{a}_1^2\mathfrak{b}_1^1+\mathfrak{a}_2^2\mathfrak{b}_2^1)
+e^{\eta_2\varphi_2+i\psi_2+\eta_1\varphi_1-i\psi_1}(\mathfrak{b}_1^2\mathfrak{b}_1^1+\mathfrak{b}_2^2\mathfrak{b}_2^1)
\Big].
\end{eqnarray*}
Using inequality \eqref{fghl}, we expand this expression as follows:
\begin{eqnarray*}
M_{1,2}&=&\frac{1}{\eta_1+\eta_2+i(\xi_1-\xi_2)}\Big[
e^{-\eta_2\varphi_2-i\psi_2-\eta_1\varphi_1+i\psi_1}\nonumber\\
&&\times
[(1+\O_{L^{\infty}}(\ep))(1+\O_{L^{\infty}}(\ep))+\O_{L^2\cap L^{\infty}}(\ep)\O_{L^2\cap L^{\infty}}(\ep)]\nonumber\\
&&-e^{\eta_2\varphi_2+i\psi_2-\eta_1\varphi_1+i\psi_1}
[\O_{L^2\cap L^{\infty}}(\ep)(1+\O_{L^{\infty}}(\ep))+(1+\O_{L^{\infty}}(\ep))\O_{L^2\cap L^{\infty}}(\ep)]\nonumber\\
&& -e^{-\eta_2\varphi_2-i\psi_2+\eta_1\varphi_1-i\psi_1}
[(1+\O_{L^{\infty}}(\ep))\O_{L^2\cap L^{\infty}}(\ep)+\O_{L^2\cap L^{\infty}}(\ep)(1+\O_{L^{\infty}}(\ep))]\nonumber\\
&&
+e^{\eta_2\varphi_2+i\psi_2+\eta_1\varphi_1-i\psi_1}
[\O_{L^2\cap L^{\infty}}(\ep)\O_{L^2\cap L^{\infty}}(\ep)+(1+\O_{L^{\infty}}(\ep))(1+\O_{L^{\infty}}(\ep))]
\Big]\nonumber\\
&=& \frac{(e^{-\eta_2\varphi_2-i\psi_2-\eta_1\varphi_1+i\psi_1}+
e^{\eta_2\varphi_2+i\psi_2+\eta_1\varphi_1-i\psi_1})}{\eta_1+\eta_2+i(\xi_1-\xi_2)}\nonumber\\
&&+\frac{(e^{-\eta_2\varphi_2-i\psi_2-\eta_1\varphi_1+i\psi_1}+
e^{\eta_2\varphi_2+i\psi_2+\eta_1\varphi_1-i\psi_1})}{\eta_1+\eta_2+i(\xi_1-\xi_2)}[\mathcal{O}_{ L^{\infty}}(\ep)+\mathcal{O}_{L^2\cap L^{\infty}}(\ep)] \\
&& -\frac{(e^{\eta_2\varphi_2+i\psi_2-\eta_1\varphi_1+i\psi_1}+
e^{-\eta_2\varphi_2-i\psi_2+\eta_1\varphi_1-i\psi_1})}{\eta_1+\eta_2+i(\xi_1-\xi_2)}\mathcal{O}_{L^2\cap L^{\infty}}(\ep),
\end{eqnarray*}
which yields
\begin{eqnarray}
\label{EstforM12}
M_{1,2} &=&M_{1,2}^S+M_{1,2}^S[\mathcal{O}_{ L^{\infty}}(\ep)+\mathcal{O}_{L^2\cap L^{\infty}}(\ep)]\\
&&-\frac{(e^{\eta_2\varphi_2+i\psi_2-\eta_1\varphi_1+i\psi_1}+
e^{-\eta_2\varphi_2-i\psi_2+\eta_1\varphi_1-i\psi_1})}{\eta_1+\eta_2+i(\xi_1-\xi_2)}\mathcal{O}_{L^2\cap L^{\infty}}(\ep),
\nonumber
\end{eqnarray}
where $\mathcal{O}_{ L^{\infty}}(\ep)$ means that the function is $\mathcal{O}(\ep)$ small in the $L^{\infty}(\mathbb{R})$ norm.
In the same way, we are able to show that
\be\label{EstforM22}
 M_{2,2}=M_{2,2}^S
+ M_{2,2}^S[\mathcal{O}_{ L^{\infty}}(\ep)+\mathcal{O}_{L^2\cap L^{\infty}}(\ep)] -\frac{(e^{-2i\psi_2}+e^{2i\psi_2})}{2\eta_2}
\mathcal{O}_{L^2\cap L^{\infty}}(\ep).
\ee
One can see that similar asymptotics hold for $M_{2,1}$ and $M_{1,1}$ which in turn yield the following expansion of $D$:
\be\label{defD}
D  =  D^S+ D^S \O_{L^{\infty}}(\ep)+(E_1+E_2) \O_{L^2\cap L^{\infty}}(\ep).
\ee
where
\begin{eqnarray}
E_1&=&\left[
\frac{1}{4\eta_1\eta_2}+\frac{1}{(\eta_1+\eta_2)^2+(\xi_1-\xi_2)^2}
\right]
\left[ e^{2\eta_2 \varphi_2-2\eta_1 \varphi_1}
+e^{-2\eta_2 \varphi_2+2\eta_1 \varphi_1}
+e^{2i\psi_2+2i\psi_1} \right. \nonumber\\
&& \left.
+e^{-2i\psi_2-2i\psi_1}
+e^{2\eta_2\varphi_2+2i\psi_1}
+e^{-2\eta_2\varphi_2-2i\psi_1}
+e^{2\eta_1\varphi_1-2i\psi_2}
+e^{-2\eta_1\varphi_1+2i\psi_2}
\right]
\end{eqnarray}
and
\begin{eqnarray}
E_2&=&
 \left[
\frac{1}{4\eta_1\eta_2}-\frac{1}{(\eta_1+\eta_2)^2+(\xi_1-\xi_2)^2}
\right]
\left[ e^{2\eta_2 \varphi_2+2\eta_1 \varphi_1}
+e^{-2\eta_2 \varphi_2-2\eta_1 \varphi_1}
+e^{2i\psi_2-2i\psi_1} \right. \nonumber\\
&& \left. +e^{-2i\psi_2+2i\psi_1}
+e^{2\eta_2\varphi_2-2i\psi_1}
+e^{-2\eta_2\varphi_2+2i\psi_1}
+e^{2\eta_1\varphi_1+2i\psi_2}
+e^{-2\eta_1\varphi_1-2i\psi_2}
\right].
\end{eqnarray}

We claim $D = D^S (1 + \O_{L^{\infty}}(\ep)),$ which has as a consequence
\be\label{1overD}
\frac{1}{D}=\frac{1}{D^S(1+\O_{L^{\infty}}(\ep))}=\frac{1+\O_{L^{\infty}}(\ep)}{D^S}.
\ee
Indeed, assuming $(\eta_1,\xi_1)\neq (\eta_2,\xi_2)$ and
letting $\gamma = 2\eta_1\eta_2 \left((\eta_1+\eta_2)^2+(\xi_1-\xi_2)^2 \right)$,
it is clear that
\begin{eqnarray}
\label{detSgrowth}
\gamma^{-1}\abs{D^S}
&=&
\left((\eta_1+\eta_2)^2+(\xi_1-\xi_2)^2 \right)
\Big(\cosh\big(2(\eta_1\varphi_1+\eta_2\varphi_2)\big)
+\cosh\big(2(\eta_1\varphi_1-\eta_2\varphi_2)\big)\Big)\nonumber\\
&&
-4\eta_1\eta_2\Big(\cosh\big(2(\eta_1\varphi_1+\eta_2\varphi_2)\big)
+\cos\big(2(\psi_1-\psi_2)\big)\Big)
\nonumber\\
&=&
\left((\eta_1-\eta_2)^2+(\xi_1-\xi_2)^2 \right)
\cosh\big(2(\eta_1\varphi_1+\eta_2\varphi_2)\big) \nonumber\\
&&
+\left((\eta_1+\eta_2)^2+(\xi_1-\xi_2)^2 \right)
\cosh\big(2(\eta_1\varphi_1-\eta_2\varphi_2)\big)
-4\eta_1\eta_2\cos\big(2(\psi_1-\psi_2)\big)
\nonumber\\
&\gtrsim &
\cosh\big(2(\eta_1\varphi_1+\eta_2\varphi_2)\big)
+\cosh\big(2(\eta_1\varphi_1-\eta_2\varphi_2)\big),
\end{eqnarray}
because  $(\eta_1+\eta_2)^2+(\xi_1-\xi_2)^2 >4\eta_1\eta_2$ and
$\cosh 2(\eta_1\varphi_1-\eta_2\varphi_2)\geq 1\geq\cos (\psi_1-\psi_2).$
Using this estimate, we obtain
\begin{equation*}
\abs{E_1}\lesssim\cosh\big(2(\eta_2\varphi_2-\eta_1\varphi_1)\big)
\lesssim \abs{D^S}.
\end{equation*}
Since
\begin{equation*}
\cosh(2\eta_1\varphi_1)+\cosh(2\eta_2\varphi_2)
\lesssim
\cosh\big(2(\eta_1\varphi_1+\eta_2\varphi_2)\big)
+\cosh\big(2(\eta_1\varphi_1-\eta_2\varphi_2)\big),
\end{equation*}
we also have
\begin{equation*}
\abs{E_2}\lesssim\cosh\big(2(\eta_1\varphi_1+\eta_2\varphi_2)\big)+\cosh(2\eta_1\varphi_1)+\cosh(2\eta_2\varphi_2)
\lesssim\abs{D^S},
\end{equation*}
from which the claim follows.

We turn to the asymptotics for $F.$
$F_{2,1}$ satisfies the following
\begin{eqnarray}\label{EstforU}
&& F_{2,1} = F_{2,1}^S+F_{2,1}^S\O_{L^{\infty}}(\ep)\\
&& \phantom{texttest} +[e^{\eta_2\varphi_2-i\psi_2+\eta_1\varphi_1+i\psi_1}
+e^{-\eta_2\varphi_2+i\psi_2-\eta_1\varphi_1-i\psi_1}-e^{\eta_2\varphi_2-i\psi_2-\eta_1\varphi_1-i\psi_1}]\O_{L^2\cap L^{\infty}}(\ep),\nonumber
\end{eqnarray}
while $F_{1,1}$ can be expanded as
\begin{eqnarray}\label{EstforT}
F_{1,1} = F_{1,1}^S+F_{1,1}^S\O_{L^{\infty}}(\ep)+[e^{2\eta_1\varphi_1}+e^{-2\eta_1\varphi_1}-e^{-2i\psi_1}]\O_{L^2\cap L^{\infty}}(\ep).
\end{eqnarray}

\vskip.4in
\underline{Step 2: Estimates on a compact set.}\\

Let $R>0$ be a large constant independent of $\e$ (and $t$) to be fixed later.
The purpose of this point is to obtain asymptotics of $F_{1,1}M_{2,2}$ and $F_{2,1}M_{1,2}$
on the compact set $\{\abs{x}<R\}$ (for all $t\in \mathbb{R}$). For our goal, it is enough
to show that these terms differ from  $F^S_{1,1}M^S_{2,2}$ and $F^S_{2,1}M^S_{1,2}$
respectively, by a quantity uniformly controlled by $\e.$

To this end, we note that on the compact set, the previous estimates \eqref{EstforM12}--\eqref{EstforT} yield
\be\label{1}
\abs{F_{1,1}M_{2,2}-F^S_{1,1}M^S_{2,2}}
\lesssim
\cosh\big(2(\eta_1\varphi_1+\eta_2\varphi_2)\big)
+\cosh\big(2(\eta_1\varphi_1-\eta_2\varphi_2)\big)
\ee
and
\be\label{2}
\abs{F_{2,1}M_{1,2}-F^S_{2,1}M^S_{1,2}}
\lesssim
\cosh\big(2(\eta_1\varphi_1+\eta_2\varphi_2)\big)
+\cosh\big(2(\eta_1\varphi_1-\eta_2\varphi_2)\big),
\ee
where $\O_{L^{\infty}}(\ep)$ is now used in the $L^{\infty}([-R,R])$ norm.

From \eqref{1overD} and \eqref{detSgrowth}, we observe that
\be
\abs{D}\gtrsim \abs{D^S}\gtrsim
\cosh\big(2(\eta_1\varphi_1+\eta_2\varphi_2)\big)
+\cosh\big(2(\eta_1\varphi_1-\eta_2\varphi_2)\big),
\ee
so that
\begin{eqnarray}\label{1<R}
\abs{\frac{F_{1,1}M_{2,2}}{D}-\frac{F^S_{1,1}M^S_{2,2}}{D^S}}
&=&
\abs{\frac{F_{1,1}M_{2,2}}{D^S}(1+\O_{L^{\infty}}(\ep))-\frac{F^S_{1,1}M^S_{2,2}}{D^S}}\nonumber\\
&\lesssim&
\abs{\frac{F_{1,1}M_{2,2}}{D^S}-\frac{F^S_{1,1}M^S_{2,2}}{D^S}}+\O_{L^{\infty}}(\ep)\abs{\frac{F_{1,1}M_{2,2}}{D^S}}\nonumber\\
&=&\O_{L^{\infty}}(\ep).
\end{eqnarray}
Similarly, one has
\be\label{2<R}
\abs{\frac{F_{2,1}M_{1,2}}{D}-\frac{F^S_{2,1}M^S_{1,2}}{D^S}}
=\O_{L^{\infty}}(\ep),
\ee
thanks to \eqref{1} and \eqref{2}.

\vskip.4in
\underline{Step 3: Estimates outside a compact set.}\\

We now deal with the estimates outside the compact set $\{\abs{x}<R\}$ for the same $R$ as in Step 2.
Here,  the norms $L^2,$ $L^{\infty}$ are taken outside the compact set $\{\abs{x}<R\}$.
Again, from
\eqref{EstforM12}--\eqref{EstforT}, we have
\begin{eqnarray}\label{N_1>R}
F_{11}M_{22} = F_{11}^SM_{22}^S & - & \frac{1}{2\eta_2}
\left( e^{-2\eta_2\varphi_2+2i\psi_1}+e^{2\eta_2\varphi_2+2i\psi_1} \right) \O_{L^{\infty}}(\ep) \\
& - & \frac{1}{2\eta_2} ( e^{-2\eta_1\varphi_1}+e^{2\eta_1\varphi_1})(e^{2i\psi_2}+e^{-2i\psi_2}) \O_{L^{\infty}}(\ep) \nonumber\\
& + & \frac{1}{2\eta_2} (e^{-2\eta_2\varphi_2}+e^{2\eta_2\varphi_2})(e^{-2\eta_1\varphi_1}+e^{2\eta_1\varphi_1}) \O_{L^2\cap L^{\infty}}(\ep),\nonumber
\end{eqnarray}
and
\begin{eqnarray}\label{N_2>R}
F_{2,1}M_{1,2} = F^S_{2,1}M^S_{1,2} + \frac{1}{\eta_1+\eta_2+i(\xi_1-\xi_2)} \left[ G_1 \O_{L^{\infty}}(\ep) +
G_2 \O_{L^2\cap L^{\infty}}(\ep) \right],
\end{eqnarray}
where
\begin{eqnarray*}
G_1 & = & -e^{-\eta_2\varphi_2+i\psi_2+\eta_1\varphi_1+i\psi_1}
(e^{-\eta_2\varphi_2-i\psi_2-\eta_1\varphi_1+i\psi_1}+e^{\eta_2\varphi_2+i\psi_2+\eta_1\varphi_1-i\psi_1}), \\
G_2 & = & e^{-\eta_2\varphi_2+i\psi_2+\eta_1\varphi_1+i\psi_1}
(e^{\eta_2\varphi_2+i\psi_2-\eta_1\varphi_1+i\psi_1}+e^{-\eta_2\varphi_2-i\psi_2+\eta_1\varphi_1-i\psi_1}) \\
& \phantom{t} & +(e^{\eta_2\varphi_2-i\psi_2+\eta_1\varphi_1+i\psi_1}
+e^{-\eta_2\varphi_2+i\psi_2-\eta_1\varphi_1-i\psi_1})
(e^{-\eta_2\varphi_2-i\psi_2-\eta_1\varphi_1+i\psi_1}+e^{\eta_2\varphi_2+i\psi_2+\eta_1\varphi_1-i\psi_1})\\
& \phantom{t} &
-e^{\eta_2\varphi_2-i\psi_2-\eta_1\varphi_1-i\psi_1}
(e^{-\eta_2\varphi_2-i\psi_2-\eta_1\varphi_1+i\psi_1}+e^{\eta_2\varphi_2+i\psi_2+\eta_1\varphi_1-i\psi_1})\\
& \phantom{t} &
+(e^{\eta_2\varphi_2-i\psi_2+\eta_1\varphi_1+i\psi_1}
+e^{-\eta_2\varphi_2+i\psi_2-\eta_1\varphi_1-i\psi_1})
(-e^{\eta_2\varphi_2+i\psi_2-\eta_1\varphi_1+i\psi_1}-e^{-\eta_2\varphi_2-i\psi_2+\eta_1\varphi_1-i\psi_1}) \nonumber\\
& \phantom{t} &
-e^{\eta_2\varphi_2-i\psi_2-\eta_1\varphi_1-i\psi_1}
(-e^{\eta_2\varphi_2+i\psi_2-\eta_1\varphi_1+i\psi_1}-e^{-\eta_2\varphi_2-i\psi_2+\eta_1\varphi_1-i\psi_1}).
\end{eqnarray*}
From these estimates, we obtain
\begin{eqnarray}\label{AAA}
\abs{F_{1,1}M_{2,2}-F^S_{11}M^S_{2,2}}&\lesssim&\mathcal{O}_{L^2\cap L^{\infty}}(\e)
+(H_1 +H_2)\mathcal{O}_{L^{\infty}}(\e)+ H_1 H_2\;\mathcal{O}_{L^2\cap L^{\infty}}(\e),\nonumber
\end{eqnarray}
where
\begin{eqnarray*}
H_1 & = & -(e^{-2\eta_1\varphi_1}+e^{2\eta_1\varphi_1})(e^{2i\psi_2}+e^{-2i\psi_2}),\\
H_2 & = & -(e^{-2\eta_2\varphi_2+2i\psi_1}+e^{2\eta_2\varphi_2+2i\psi_1}).
\end{eqnarray*}
We can see these functions satisfy the bounds
\begin{equation}\label{Hibounds}
\abs{H_i}\lesssim \cosh(2\eta_i \varphi_i),\quad\mbox{ for }i=1,2.
\end{equation}
In the same way, we have
\begin{eqnarray}
\abs{F_{2,1}M_{1,2}-F^S_{2,1}M^S_{1,2}} & \lesssim & J_1 J_3 \mathcal{O}_{L^{\infty}}(\e) \\
& \phantom{t} & + (J_1 J_4 +J_3^2+J_2 J_3+J_3J_4+J_2J_4) \mathcal{O}_{L^2 \cap L^{\infty}}(\e),\nonumber
\end{eqnarray}
where
\begin{eqnarray*}\label{Jsdef}
J_1 & = & -e^{-\eta_2\varphi_2+\eta_1\varphi_1},  \\
J_2 & = & -e^{\eta_2\varphi_2-\eta_1\varphi_1}, \\
J_3 & = & e^{-\eta_2\varphi_2-\eta_1\varphi_1} +e^{\eta_2\varphi_2+\eta_1\varphi_1}, \\
J_4 & = & -e^{\eta_2\varphi_2-\eta_1\varphi_1}-e^{-\eta_2\varphi_2+\eta_1\varphi_1}.
\end{eqnarray*}

It is immediate from \eqref{Hibounds} that
\be\label{a}
\abs{H_1 H_2}\lesssim
\cosh (2(\eta_1\varphi_1+\eta_2\varphi_2))
+\cosh (2(\eta_1\varphi_1-\eta_2\varphi_2)).
\ee

In a similar fashion, we deduce from \eqref{Jsdef} that
\begin{eqnarray*}
&&\abs{J_1 J_4}\lesssim e^{-2\eta_2\varphi_2}+e^{2\eta_1\varphi_1},\nonumber\\
&&\abs{J_2 J_4}\lesssim
1+e^{2\eta_2\varphi_2-2\eta_1\varphi_1},\nonumber\\
&&\abs{J_3^2}\lesssim\cosh(2\eta_2\varphi_2+2\eta_1\varphi_1)+1,\nonumber\\
&&\abs{J_3 J_4}\lesssim
\cosh(2\eta_1\varphi_1)+\cosh(2\eta_2\varphi_2).
\end{eqnarray*}
All these quantities can be bounded from above by $\abs{D},$ in light of \eqref{1overD} and \eqref{detSgrowth}.
Thus, we can assert that
$\frac{H_1 H_2}{D}, \frac{J_1 J_4}{D},$  $\frac{J_3^2}{D}, \frac{J_2 J_4}{D}$ and $\frac{J_3 J_4}{D}$ are $L^{\infty}$ functions whose norms are bounded uniformly in $t.$

Now we show that $\frac{H_1}{D}, \frac{H_2}{D}, \frac{J_1 J_3}{D}$ and $\frac{J_2 J_3}{D}$ are bounded in $L^2,$ uniformly in $t.$
From \eqref{a} and similar estimates, we see that it suffices to prove that $\frac{e^{2\eta_1\varphi_1}}{D}\in L^2$  (the terms $\frac{e^{-2\eta_1\varphi_1}}{D},$  $e^{\pm 2\eta_2\varphi_2}$ can be dealt with similarly).

Appealing once again to \eqref{1overD} and \eqref{detSgrowth} we see that
\begin{eqnarray}
\frac{e^{-2\eta_1\varphi_1}}{\abs{D}}&\lesssim&
\frac{e^{-2\eta_1\varphi_1}}{\cosh (2(\eta_1\varphi_1+\eta_2\varphi_2))
+\cosh (2(\eta_1\varphi_1-\eta_2\varphi_2))}\nonumber\\
&\lesssim&
\frac{e^{-2\eta_1\varphi_1}}{e^{2\eta_2\varphi_2+2\eta_1\varphi_1}+e^{-2\eta_2\varphi_2+2\eta_1\varphi_1}}
\lesssim\mbox{sech}(2\eta_2\varphi_2)\in L^2.
\end{eqnarray}

Collecting all these estimates, we see that outside the compact set $\{\abs{x}<R\}$
\be\label{1>R}
\frac{F_{1,1}M_{2,2}}{D}=\frac{F_{1,1}^S M_{2,2}^S}{D^S}+\mathcal{O}_{L^2}(\e)
\ee
and
\be\label{2>R}
\frac{F_{1,2}M_{1,2}}{D}=\frac{F_{1,2}^S M_{1,2}^S}{D^S}+\mathcal{O}_{L^2}(\e).
\ee
Combining \eqref{1<R}, \eqref{2<R}, \eqref{1>R} and \eqref{2>R}, we obtain \eqref{normQ-Q_S}.
\end{proof}

\begin{corollary}\label{PPP-cor}
Under conditions of Proposition \ref{PPP}, there is a positive constant $C$ such that
\be\label{normQ}
\| q_0 \|_{L^2} \leq C \norm{q-q^S}_{L^{2}}, \quad t \in \mathbb{R}.
\ee
\end{corollary}

\begin{proof}
Because the dressing transformation formulas (\ref{dressing-trans}) and (\ref{transformation-q})
are invertible by the construction, the bound (\ref{normQ}) follows  from the triangle inequality
and bound (\ref{bound-tech}).
\end{proof}

\begin{rem}
Proposition \ref{PPP} can be extended to the general case of multi-soliton configurations.
The global estimates obtained in \eqref{EstforM12}--\eqref{EstforT} can be used
to derive explicit (though cumbersome) expansions for the $D_{j,k}$'s
(including the determinant $D$). On the compact set $\{\abs{x}<R\}$,
the $\mathcal{O}_{L^{\infty}}(\varepsilon)$ difference between $D_{j,k} F_{j,k}$'s
and   $D_{j,k}^{S} F_{j,k}^{S}$'s together with
\[
\frac{1}{D}=\frac{1}{D^S}(1+\mathcal{O}_{L^{\infty}(\varepsilon)})
\]
suffices to show
$$
\frac{\Sigma}{D}=\frac{\Sigma^S}{D^S}+\mathcal{O}_{L^{\infty}}(\varepsilon)
$$
for all $t\in \mathbb{R}$. The estimates outside the compact set $\{\abs{x}<R\}$
can be achieved as in the third step, thanks to the bound
\be
e^{2\left(\sum_{i=1}^n \sigma_i\eta_i\varphi_i\right)}\lesssim\abs{D},
\ee
valid for any choice of signs $\sigma:\{1,\ldots,n\}\to\{-1,1\},$ and making use of the fact that for any $j\in\{1,\ldots,n\},$
\be
\frac{e^{2\sum_{i\neq j}\sigma_i\eta_i\varphi_i}}{\abs{D}}\lesssim
\mbox{sech }2\eta_j\varphi_j,
\ee
which is in $L^2$ with norm bounded independent of $t.$
The estimates thus obtained can be seen to be independent of the time,
thanks to the conservation of the $L^2(\mathbb{R})$ norm of the solution $q_0$
of the cubic NLS equation (\ref{NLS}).
\end{rem}

\section{Orbital stability of multi-solitons}

In this section we prove the result on the orbital stability of multi-solitons in $L^2(\mathbb{R})$
given by Theorem \ref{theorem-main}. First, we will assume that the initial data $q_0$
for the cubic NLS equation (\ref{NLS}) satisfies the bound (\ref{closeness}) (that is, it is close
the multi-soliton $q^S$) and belongs to $H^3(\mathbb{R})$. By the well-posedness theory
for the NLS equation \cite{GV,Kato}, there exists a unique solution in class
$$
q \in C^0(\mathbb{R},H^3(\mathbb{R})) \cap C^1(\mathbb{R},H^1(\mathbb{R})),
$$
such that $q|_{t = 0} = q_0$. By Sobolev embeddings, $q_t$ and $q_{xx}$ are continuous
functions of $(x,t) \in \mathbb{R} \times \mathbb{R}$ such that $q$ is a classical solution
of the cubic NLS equation (\ref{NLS}) in $L^2(\mathbb{R})$.

By Proposition \ref{PPP}, a small $L^2$-neighborhood of the zero solution of the cubic NLS equation
(\ref{NLS}) is mapped into a small $L^2$-neighborhood of a multi-soliton $q^S$ (details were given
for the case of $2$-solitons). The dressing transformation formulas (\ref{dressing-trans}) and (\ref{transformation-q})
are invertible by the construction. However,
since the dressing transformation is not onto, an arbitrary
point in a small $L^2$-neighborhood of the multi-soliton $q^S$ is not mapped back
to the small $L^2$-neighborhood of the zero solution, unless constraints are
set to specify uniquely the parameters of the multi-soliton $q^S$ \cite{MP}.

To avoid the lengthy analysis of \cite{MP} with decomposition of multi-solitons using
the symplectic orthogonality constraints, we apply the inverse scattering transform methods.
Results of the inverse scattering transform methods for the cubic NLS equation
are collected together in \cite{CP}.
Given the initial data $q_0$ near a multi-soliton $q^S$ in the weighted $L^2(\mathbb{R})$
space according to the bound (\ref{closeness}), the direct and inverse scattering problems
can be solved as in \cite{CP} to obtain parameters of the multi-soliton $q^{S'}$
from the initial data $q_0$. Thanks to the bound (\ref{closeness}), the initial data
$q_0$ supports exactly $n$ eigenvalues in the Lax system (\ref{eq:zs}) if $q^S$
supports $n$ eigenvalues.

Here we note two important facts. First, multi-solitons of the cubic NLS equation
belong to the class of generic potentials in $L^1(\mathbb{R})$, which means that
a small perturbation to a multi-soliton in $L^1(\mathbb{R})$ does not change
the number of solitons (eigenvalues of the Lax operators). This applies to
a small perturbation in $L^{2,s}(\mathbb{R})$ with $s > \frac{1}{2}$,
which is continuously embedded into $L^1(\mathbb{R})$.
Second, the scattering data associated with $q_0$ are Lipschitz continuous in $H^s(\mathbb{R})$
if $q_0 \in L^{2,s}(\mathbb{R})$ with $s > \frac{1}{2}$ (Lemma 2.4 in \cite{CP}).
Bound (\ref{bounds-1}) follow from Lipschitz continuity of the scattering data
associated with the initial data $q_0$ satisfying the bound (\ref{closeness}).

By using the inverse dressing transformation (\ref{dressing-trans}) and
\eqref{transformation-q} with parameters of
the multi-soliton $q^{S'}$ chosen from the inverse scattering transform associated
with the potential $q_0$, we map the initial data $q_0$ to the new initial data
$\tilde{q}_0$, which is free of solitons. By Corollary \ref{PPP-cor},
it satisfies the bound
$$
\| \tilde{q}_0 \|_{L^2} \leq C \epsilon
$$
for some $C > 0$, where $\epsilon$ is defined by the initial bound (\ref{closeness}).

Let $\tilde{q}$ be a classical solution of the cubic NLS equation (\ref{NLS}) in $L^2(\mathbb{R})$
such that $\tilde{q} |_{t = 0} = \tilde{q}_0$.
By the $L^2$ conservation, we have $\| \tilde{q} \|_{L^2} \leq C \epsilon$
for all $t \in \mathbb{R}$. Then, by the dressing transformation (\ref{dressing-trans}) and
\eqref{transformation-q} with the same parameters and Proposition \ref{PPP}, we
obtain the bound (\ref{bounds-2}) for all $t \in \mathbb{R}$.

To complete the proof of Theorem \ref{theorem-main}, we need to show that the bound
(\ref{bounds-2}) remains true if $q_0$ satisfies the bound (\ref{closeness}) but does not
belong to $H^3(\mathbb{R})$ (Bound (\ref{bounds-1}) remains true, thanks
to the inverse scattering transform results \cite{CP}.) In this case,
$q_0$ generates a global solution of the cubic NLS equation (\ref{NLS})
in class $q \in C(\mathbb{R},L^2(\mathbb{R}))$, which is not a classical solution of the NLS equation.
Therefore, the dressing transformation (\ref{dressing-trans}) and
\eqref{transformation-q} cannot be used directly
for the solution $q$. Instead, we consider an approximating sequence in Sobolev
spaces, similarly to \cite{CP,MP}.

Let $\{ q^{(k)}_0 \}_{k \in \mathbb{N}}$ be a sequence in $H^3(\mathbb{R})$
such that $q_0^{(k)} \to q_0$ in $L^2(\mathbb{R})$
as $k \to \infty$. Then $\{ q^{(k)} \}_{k \in \mathbb{N}}$ is a sequence of
classical solutions of the cubic NLS equation such that $q^{(k)} |_{t = 0} = q^{(k)}_0$.
By the previous arguments, there exists a sequence of $n$-soliton solution $q^{S^{(k)}}$ with parameters
$\{ {\xi}^{(k)}_j, {\eta}^{(k)}_j, {x}^{(k)}_j, {\theta}^{(k)}_j \}_{j = 1}^n$
such that
$$
\max_{1 \leq j \leq n} \left| ({\xi}^{(k)}_j,{\eta}^{(k)}_j,{x}^{(k)}_j,{\theta}^{(k)}_j) -
(\xi_j, \eta_j, x_j, \theta_j) \right| \leq C \epsilon
$$
and
$$
\| q^{(k)}(\cdot,t) - q^{S^{(k)}}(\cdot,t) \|_{L^2(\mathbb{R})} \leq C \epsilon, \quad t \in \mathbb{R}.
$$
Thanks to the $L^2(\mathbb{R})$ conservation, the sequence
$\{ q^{(k)} \}_{k \in \mathbb{N}}$ converges to $q$ in the $L^2(\mathbb{R})$ norm
as $k \to \infty$ for any $t \in \mathbb{R}$. As a result,
there is a subsequence that converges the $n$-soliton $q^{S'}$ with
parameters $\{ {\xi}'_j, {\eta}'_j, {x}'_j, {\theta}'_j \}_{j = 1}^n$
such that bounds (\ref{bounds-1}) and (\ref{bounds-2}) are satisfied. The proof of Theorem \ref{theorem-main}
is complete.

\end{document}